\documentclass[leqno,11pt]{amsart}
\usepackage{amsmath,amssymb}
\usepackage{color}
\usepackage{comment}

\usepackage{amsfonts}      
\usepackage{amsthm}         
\usepackage{bbding}         
\usepackage{bm}             
\usepackage{graphicx}   
\usepackage{fancyvrb}       
\usepackage{ragged2e}
			    
\usepackage{dcolumn}        
\usepackage{booktabs}       
\usepackage{paralist}       
\usepackage{xcolor}        
\usepackage[dvipsnames]{xcolor}

\usepackage{amssymb}
\usepackage{units}
\usepackage{mathtools}
\usepackage[colorlinks=true,linkcolor=blue]{hyperref}
\hypersetup{colorlinks=true,  linkcolor=blue, citecolor=green, urlcolor=cyan}

\allowdisplaybreaks

\theoremstyle{plain}
\newtheorem{thm}{Theorem}[section]

\newtheorem{prop}[thm]{Proposition}
\newtheorem{cor}[thm]{Corollary}

\theoremstyle{definition}
\newtheorem{defn}[thm]{Definition}
\newtheorem{rem}[thm]{Remark}

\newcommand{\Nat}{\mathbb{N}}

\newcommand{\R}{\mathcal{R}}
\newcommand{\M}{\mathcal{M}}
\renewcommand{\S}{\mathcal{S}}
\renewcommand{\d}{\,{\rm d}}

\newcommand{\emb}{\hookrightarrow}

\newcommand{\supp}{\rm{supp}\ }

\newcommand{\abs}[1]{\left|{#1}\right|}

\newcommand{\medint}{-\kern  -,435cm\int}
\newcommand{\medintinrigo}{-\kern  -,315cm\int}
\newcommand{\medelle}{-\kern  -,235cm L}
\newcommand{\medellenrigo}{-\kern  -,180cm L}

\newtoks\by
\newtoks\paper
\newtoks\book
\newtoks\jour
\newtoks\yr
\newtoks\pages
\newtoks\vol
\newtoks\publ
\newtoks\eds
\newtoks\proc
\newtoks\no
\def\ota{{\hbox{???}}}
\def\cLear{\by=\ota\paper=\ota\book=\ota\jour=\ota\yr=\ota
\pages=\ota\vol=\ota\publ=\ota}
\def\endpaper{\the\by, \textit{\the\paper},
{\the\jour} \textbf{\the\vol} (\the\yr), \the\pages.\cLear}
\def\endbook{\the\by, \textit{\the\book}, \the\publ.\cLear}
\def\endprep{\the\by, \textit{\the\paper}, \the\jour.\cLear}
\def\endproc{\the\by, \textit{\the\paper}, \the\publ, \the\pages.\cLear}

\numberwithin{equation}{section}

\newcommand{\norm}[1]{{\left\vert\kern-0.25ex\left\vert\kern-0.25ex\left\vert
#1
    \right\vert\kern-0.25ex\right\vert\kern-0.25ex\right\vert}}

\usepackage[normalem]{ulem}
\usepackage{soul}
\usepackage{cancel}
\newcommand\Del[1]{{\color{red}\ifmmode\cancel{#1}\else\sout{#1}\fi}}

\usepackage[
	textsize=tiny,
	textwidth=3.2cm,
	colorinlistoftodos,
	backgroundcolor=teal!30!white,
	linecolor=teal!50!white
	]{todonotes}

\thanks{This research was supported in part by the grant 23-04720S of the Czech Science Foundation, grant UNCE24/SCI/005 of Charles University, GA UK no. 148125 of Charles University and grant SVV-2025-260827}

\begin{document}

\date{\today}

\title{Nonstandard Calderón-type theorems}

\author {David Kub\'i\v cek}

\address{David Kub\'i\v cek, Department of Mathematical Analysis\\
Faculty of Mathematics and Physics\\
Charles University\\
Sokolovsk\'a~83\\
186~00 Praha~8\\
Czech Republic}
\email{kubicek@karlin.mff.cuni.cz}
\urladdr{\href{https://orcid.org/0009-0003-4822-5268}{\texttt{0009-0003-4822-5268}}}



\begin{abstract}
We establish Calderón-type theorems for operators bounded on nonstandard end-point Lorentz spaces \begin{equation*} T\colon L^{p_0, q_0}\to L^{p_1, q_1}\quad\text{and}\quad T\colon L^{q, 1}\to L^\infty \end{equation*} and the improvement of target spaces which is intimately connected with this. The emphasis will be placed on the cases $q_0=q_1$ and $q_1=\infty$.
\end{abstract}

\maketitle

\section{Introduction}
Let $(\R, \mu)$ and $(\S, \nu)$ be non-atomic $\sigma$-finite measure spaces. The standard Marcinkiewicz interpolation theorem, first introduced by J. Marcinkiewicz in 1939 and later generalized by A. Zygmund in 1956 \cite{Zy56}, asserts that if a quasilinear operator $T$ satisfies a pair of end-point estimates, 
\begin{equation}\label{eq:Marci}
    T\colon L^{p_0, 1}(\R, \mu)\to L^{q_0, \infty}(\S, \nu)\quad\text{and}\quad T\colon L^{p_1, 1}(\R, \mu)\to L^{q_1, \infty}(\S, \nu),
\end{equation}
where $1\leq p_0<p_1\leq\infty$ and $1\leq q_0, q_1\leq \infty, q_0\neq q_1$, then we can conclude that $T\colon L^{p, r}(\R, \mu)\to L^{q, r}$ for every $r\in[1 ,\infty]$ and every $p$ and $q$ which are defined by 
\begin{equation*}
    \frac{1}{p}=\frac{1-\theta}{p_0}+\frac{\theta}{p_1}\quad\text{and}\quad \frac{1}{q}=\frac{1-\theta}{q_0}+\frac{\theta}{q_1}
\end{equation*}
for some $\theta\in(0, 1)$. Here $L^{p, q}$ stands for standard Lorentz space, see \eqref{eq:LorentzSpace} for precise definition of its governing functional. Operators which satisfy \eqref{eq:Marci} are said to be of weak type $(p_0, q_0)$ and $(p_1, q_1)$. This result was furthermore generalized to what is now known as Calderón's fundamental interpolation theorem in \cite{Ca1966}. 
In essence, it describes the behaviour of operators satisfying end-point estimates \eqref{eq:Marci} on general rearrangement-invariant spaces in terms of boundedness of the so-called \emph{Calderón operator} on their representation spaces. The Calderón operator is intrinsically connected to \eqref{eq:Marci} and is defined by
\begin{equation*}
    S_mf(t)=t^{-\frac{1}{q_0}}\int_0^{t^m} s^{\frac{1}{p_0}-1}f(s)\d s+t^{-\frac{1}{q_1}}\int_{t^m}^\infty s^{\frac{1}{p_1}-1}f(s)\d s,
\end{equation*}  where $m=\frac{\frac{1}{q_0}-\frac{1}{q_1}}{\frac{1}{p_0}-\frac{1}{p_1}}$ is the \emph{slope} of the corresponding interpolation segment. 
Using Peetre's $K$-functional formula and Holmstedt's formulae, one can show that for any quasilinear operator $T$ satisfying \eqref{eq:Marci}, we have $\left(Tf\right)^*(t)\lesssim S_mf^*(t)$ for every $t>0$, a key estimate for establishing fine boundedness properties of $T$ on various function spaces, where the star denotes the operation $f\mapsto f^*$ of the non-increasing rearrangement of a given function (see definitions below). It should be noted that, in the classical situation, the Calder\'on operator is a sum of two integral type operators involving power weights. A unified treatise of weak-type operators was conducted for example in \cite{BeRu80}. 

Calderón's interpolation theorem has since been adapted into various nonstandard settings. This was motivated by the fact that several operators of great importance such as, for instance, the fractional maximal operator or the weighted Copson-type operator governing via a certain reduction principle optimal Sobolev embeddings have (different) nonstandard endpoint behaviours. Calder\'{o}n-type theorems for this sort of operators were established in \cite{GoPi2009}, in which one pair of weak-type estimates has been substituted by either boundedness from $L^{1}$ to $L^{p, 1}$ or from $L^{p, \infty}$ to $L^\infty$. A considerably more general situation, albeit using similar approach, was considered in~\cite{Mal:12}. Completely different techniques were used in \cite{BaGoMiPi22}, where a Calderón-type theorem was established for a special integral operator that reflects interpolation properties and embeddings of Gaussian-Sobolev spaces into Orlicz-Karamata spaces.

The results of~\cite{GoPi2009} revealed a surprising phenomenon. For operators that take boundedly $L^{p_0, 1}$ to $L^{q_0, \infty}$ and simultaneously $L^{p_1, \infty}$ to $L^\infty$, the Calder\'{o}n operator is once again written in the form of two suboperators, but this time one of the summands is a supremum-type operator (hence, remarkably, not linear). More precisely, the Calder\'on operator in these circumstances reads as
\begin{equation*}
    R_mf(t)=t^{-\frac{1}{q_0}}\int_0^{t^m} s^{\frac{1}{p_0}-1}f(s)\d s
    +
    \sup_{t^m\leq s<\infty}
    t^{\frac{1}{p_1}}f(t).
\end{equation*}
However, rather unexpectedly, it was shown that the latter operator has no essential meaning and can be dropped from the formula, since, when wrapped in a rearrangement-invariant norm, it is bounded above by the same norm of the former term. Consequently, via an application of the Hardy--Littlewood--P\'olya relation,  the suitable Calder\'on operator consists merely of the first summand. Needless to say that such phenomenon has absolutely \emph{no analogue} for operators having standard endpoint behaviour. 

It has recently been observed in \cite{CiPiSl20201, CiPiSl20202} while investigating trace embedding theorems and Sobolev embeddings involving Frostman measures that when one of the target spaces is $L^{\infty}(\S, \nu)$, one can, in fact, enhance the information obtained from interpolation. This improvement is captured by utilizing a structure denoted as $Y^{\langle p \rangle}(\S, \nu)$ where $Y(\S, \nu)$ is a rearrangement-invariant space. More precisely, these spaces are defined as the collection of all $\nu$-measurable functions $f$ satisfying 
\begin{equation}
    \|f\|_{Y^{\langle p \rangle}(\S, \nu)}\coloneqq\left\|t^{-\frac{1}{p}}\left(\int_0^t f^*(s)^p\d s\right)^\frac{1}{p}\right\|_{\bar Y(0,\nu(\S))}<\infty.
\end{equation} Detailed treatise of these spaces has been carried out in \cite{Tu23}. What is of importance to us is that these $Y^{\langle p\rangle}(\R, \mu)$ spaces continuously embed into $Y(\R, \mu)$. Moreover, these inclusions may be strict, and hence we gain an improvement of the information obtained through interpolation.

A Calderón-type interpolation theorem utilizing the structure of $Y^{\langle p\rangle}(\R, \mu)$ spaces has recently been established in \cite{MiPiSp}. Therein, the operators in question satisfy 
\begin{equation}\label{eq:Saw}
    T\colon L^{p}(\R, \mu)\to L^{p}(\S, \nu)\quad\text{and}\quad T\colon L^{q, \infty}(\R, \mu)\to L^{\infty}(\S, \nu).
\end{equation}
This research was motivated by the recent result of~\cite{Kor:24}, in which borderline case of a potential operator was investigated in connection with the Luzin $N$-property and the Morse-Sard property. The respective Calderón operator for operators satisfying \eqref{eq:Saw} is defined by 
\begin{equation}\label{eq:Srf}
    S_rf(t)=\left(\frac{1}{t}\int_0^{t^r}f^*(s)^p\d s\right)^\frac{1}{p}+\sup_{t^r\leq s<\infty}s^\frac{1}{q}f^*(s) \quad\text{for $t\in(0,\infty)$,}
\end{equation} 
for some $r>0$. Once again, due to the presence of the space $L^{q, \infty}(\R, \mu)$ on the domain side, we can omit the component acting around infinity. In result, the operator $S_r$ is bounded if and only if the map \begin{equation*}
    f\mapsto\left(\frac{1}{t}\int_0^{t^r}f^*(s)^p\d s\right)^\frac{1}{p}
\end{equation*} is.

In this paper, we will deal with operators enjoying the following mapping properties
\begin{equation}\label{eq:tpq}
    T\colon L^{p_0, q_0}(\R, \mu)\to L^{p_1, q_1}(\S, \nu)\quad\text{and}\quad T\colon L^{q, 1}(\R, \mu)\to L^\infty(\S, \nu)
\end{equation} 
where 
\begin{equation}\label{eq:parameters}
    1<p_0<q<\infty, \quad p_1\in(1, \infty)\quad \text{and}\quad 1\leq q_0\leq q_1\leq\infty.
\end{equation}
The motivation for studying such operators stems from the recent research of the action of the fractional maximal operator and the Riesz potential on spaces endowed with the capacitary measures and the Hausdorff content and their important applications, see~\cite{Kor:24,MiPiSp} and the references therein.

As the discussion later in this paper will make it clear, in our setting we will not be able to omit one of the Calderón suboperators via standard means as in the two examples we just discussed.
These end-point spaces will result in a new structure $Y^{\langle p, q_1\rangle}(\S, \nu)$, which generalize $Y^{\langle p\rangle}(\S, \nu)$, and are defined as the collection of all $\nu$-measurable functions satisfying 
\begin{equation*}
    \|f\|_{Y^{\langle p, q_1\rangle}(\S, \nu)}=\left\|t^{-\frac{1}{p}}\left(\int_0^t s^{\frac{q_1}{p}-1}f^*(s)^{q_1}\d s\right)^\frac{1}{q_1}\right\|_{\bar Y(0, \nu(\S))}<\infty.
\end{equation*}
The emphasis will be placed on the case where $q_0=q_1$ or $q_1=\infty$ as in this setting we will manage to establish a Calderón-type theorem. Afterwards, we will briefly comment on the case $q_0<q_1<\infty$ and the problems which arise herein.

The structure of the paper follows.

In the second section, we recall key notions, most importantly, rearrangement-invariant Banach function spaces, and fix conventions.

In the third section, we define $Y^{\langle p, q_1\rangle}(\S, \nu)$ as a natural generalization of $Y^{\langle p \rangle}(\S, \nu)$ for the setting of Lorentz spaces. We characterize the parameters for which the structure $Y^{\langle p, q_1\rangle}(\S, \nu)$ constitutes a rearrangement-invariant space, and establish certain nesting properties. 
Then we characterize operators fullfiling embeddings \eqref{eq:tpq} via $K$-functional and derive appropriate Calderón operators, $R_{q_0}$ and $S$, defined as
\begin{equation*}
    R_{q_0}f(t)= t^{-\frac{1}{p_1}}\left(\int_0^{t^r}s^{\frac{q_0}{p_0}-1}f^*(s)^{q_0}\d s\right)^\frac{1}{q_0}\quad
    \text{and}\quad
    Sf(t)=\int_{t^r}^\infty s^{\frac{1}{q}-1}f^*(s) \d s,
    \end{equation*}
where $r=\frac{p_0q}{p_1(q-p_0)}$. 
Finally, we characterize pairs of Lorentz spaces $L^{r_i, s_1}(0, \infty), i=1, 2$ for which we have $R_{q_0}\colon L^{r_1, s_1}\to L^{r_2, s_2}$ and $S\colon L^{r_1, s_1}\to L^{r_2, s_2}$ as it will both be useful in the fourth section and will provide us a~source of examples for potential applications. In the case of $R_{q_0}$, we achieve this by carefully translating the boundedness of $R_{q_0}$ into the calculation of $\left( L^{r_i, s_i}\right)^{\langle p, q_0\rangle}$.

The fourth and main section will then state and prove several Calderón-type theorems depending on the values of $q_0$ and $q_1$. In the first half of the section, we deal with the case $q_0=q_1$. Here we prove the main theorem, and, as an interesting consequence, we will obtain an extrapolation theorem.
The other half of the section will be devoted to the case $q_0<q_1$. This in turn will be split into two subparts. We will first deal with the case $q_1=\infty$, where we obtain a desirable Calderón-type theorem similar to that in the first half of the section. Afterwards, we make an alternative statement for the remaining case $q_0<q_1<\infty$.

\section{Preliminaries}
Throughout the paper, we write $A \lesssim B$ if $A$ is dominated by a constant multiple of $B$, independent of all quantities involved; these quantities will be in most cases evident from the context, occasionally, in cases of a~risk of confusion, we will specify them. By $A\approx B$ we mean that both $A \lesssim B$ and $A \gtrsim B$.

Let us recall essentials from the theory of \\ rearrangement-invariant Banach function spaces. For more details, see \cite{BS}.

Let $(\R, \mu)$ and $(\S, \nu)$ be non-atomic $\sigma$-finite measure spaces. We set 
\begin{align*}
	\M(\R, \mu)&=\{f\colon R\to [-\infty, \infty]: f\text{ is }\mu\text{-measurable in }\Omega\},\\
	\M_+(\R, \mu)&=\{f\in \M(\R, \mu): f\geq 0 \ \mu\text{-a.e.}\}.
\end{align*} 
If there is no risk of confusion we simply write $\M(\R)$ and $\M_+(\R)$. A real half-line $(0, \infty)$ equipped with the Lebesgue measure will be of particular interest to us and in this case we shall write $\M(0, \infty)$.

Given $f\in\M(\R)$ we define \emph{non-increasing rearrangement} of $f$, denoted $f^*$, by
\begin{equation*}
	f^*(t)=\inf\{ \lambda\geq 0: \mu(\{\abs{f}>\lambda\})\leq t\},\quad t\in(0, \infty).
\end{equation*} 
The non-increasing rearrangement is monotone, i.e.
\begin{equation*}
    \abs{f}\leq\abs{g}\ \mu\text{-a.e. implies }f^*\leq g^*.
\end{equation*} 
We further have \emph{Hardy-Littlewood inequality}
\begin{equation*}
    \int_\R \abs{f(x)g(x)}\d \mu(x)\leq\int_0^\infty f^*(t)g^*(t)\d t,\quad f, g\in\M(\R).
\end{equation*} 
The operation $f\mapsto f^*$ is not subadditive. However, it satisfies a weaker condition 
\begin{equation*}
    (f+g)^*(t_1+t_2)\leq f^*(t_1)+g^*(t_2),\quad t_1, t_2>0.
\end{equation*} 
We also have 
\begin{equation*}
    (fg)^*(t_1+t_2)\leq f^*(t_1)g^*(t_2),\quad t_1, t_2>0.
\end{equation*} 
If we pass to the so called \emph{maximal non-increasing rearrangement} defined by 
\begin{equation*}
	f^{**}(t)=\frac{1}{t}\int_0^t f^*(s)\d s,\quad f\in\M(\R), t\in (0, \infty),
\end{equation*} 
we obtain subadditivity. More precisely, we have $(f+g)^{**}\leq f^{**}+g^{**}$. Also note that $f^*\leq f^{**}$ for every $f\in\M(\R)$.

Having defined the non-increasing rearrangement, we can at last provide the central definition of the paper.

\begin{defn}
A mapping $\rho\colon \M_+(\R)\to [0, \infty]$ is called a \emph{rearrangement-invariant Banach function norm} if for all $f, g, f_n\in \M_+(\R), n\in\Nat$ and $a\geq 0$ the following holds:
\begin{enumerate}[(P1)]
	\item $\rho(f)=0\iff f=0\ \mu\text{-a.e.}$,\quad $\rho(af)=a\rho(f)$,\quad $\rho(f+g)\leq \rho(f)+\rho(g)$,
	\item $f\leq g\ \mu$-a.e. $\implies$ $\rho(f)\leq \rho(g)$,
	\item $f_n\nearrow f\ \mu\text{a.e. }\implies \rho(f_n)\nearrow \rho(f)$,
	\item $\rho(\chi_E)<\infty$ whenever $\mu(E)<\infty$,
	\item $\|f\|_{L^1(E)}\leq C_E\rho(f)$ for all $E\subset \R$ $\mu$-measurable with $\mu(E)<\infty$,
	\item $f^*=g^*\implies \rho(f)=\rho(g)$.
\end{enumerate}
\end{defn}
Given such a $\rho$, we define its \emph{associate functional} $\rho'$ by 
\begin{equation*}
    \rho'(g)=\sup_{\rho(f)\leq 1}\int_\R f(x)g(x)\d \mu(x),\quad g\in\M(\R).
\end{equation*} 
It then holds that $\rho'$, too, is a rearrangement-invariant norm and it obeys the \emph{principle of duality} $$\rho''\coloneqq (\rho')'=\rho.$$ \emph{Hardy-Littlewood-Pólya} principle asserts that 
\begin{equation}\label{eq:HLP}
	f^{**}\leq g^{**}\implies \rho(f)\leq\rho(g).
\end{equation}
We define space $X=X(\rho)$ as the collection of all $f\in\M(\R)$ such that $\rho(\abs{f})<\infty$. Denoting $\|f\|_X=\rho(\abs{f})$, we have that $(X, \|\cdot\|_X)$ is a Banach space. We refer to $X$ as a \emph{rearrangement-invariant Banach function space}, or briefly, a rearrangement-invariant space. Similarly, we denote $X'$ the rearrangement-invariant space corresponding to $\rho'$ and call it the \emph{associate space} of $X$.
Given rearrangement-invariant space $X$, its \emph{fundamental function} $\varphi_X\colon [0, \mu(\R))\to [0, \infty)$ is defined by 
\begin{equation*}
    \varphi_X(t)=\|\chi_E\|_X,\quad\mu(E)=t.
\end{equation*}

Let now $X(\R, \mu)$ and $Y(\S, \nu)$ be two rearrangement-invariant spaces and $T\colon \M_+(\R)\to \M_+(\S)$ be an operator. We say that $T$ is bounded from $X$ to $Y$ if 
\begin{equation*}
    \|Tf\|_Y\lesssim \|f\|_X,\quad f\in X,
\end{equation*} 
and denote this fact by $T\colon X\to Y$. If $X=Y$, we say that $T$ is bounded on $X$. Inclusions between rearrangement-invariant spaces are always continuous, i.e. 
\begin{equation*}
    X\subset Y\iff \operatorname{Id}\colon X\to Y.
\end{equation*} 
We say that an operator $T'$ on $\M_+(0, \infty)$ is an \emph{associate operator} of $T$ if
\begin{equation*}
    \int_0^\infty (Tf)(t)g(t)\d t=\int_0^\infty f(t)(T'g)(t)\d t,\quad f, g\in\M_+(0, \infty).
\end{equation*} 
We then have that 
\begin{equation*}
    T\colon X(0, \infty)\to Y(0, \infty)\iff T'\colon Y'(0, \infty)\to X'(0, \infty).
\end{equation*}

We will mainly be interested in the Lorentz $L^{p, q}(\R, \mu)$ spaces. To this end, let us recall their definition. Given $p, q\in (0, \infty]$, the Lorentz functional is defined by 
\begin{equation}\label{eq:LorentzSpace}
	\|f\|_{L^{p, q}(\R, \mu)}=\begin{cases} \left(\int_0^\infty t^{\frac{1}{p}-\frac{1}{q}}f^*(t)^q\d t\right)^\frac{1}{q},& q<\infty,\\ \sup_{t>0}t^\frac{1}{p}f^*(t),&q=\infty.\end{cases} 
\end{equation} 
The Lorentz space $L^{p, q}(\R, \mu)$ is \emph{equivalent} to a rearrangement-invariant space, in the sense that there exists a rearrangement-invariant norm $\|\cdot\|$ such that $\|f\|\approx\|_{p, q}$ for every $f\in\M(\R, \mu)$, if and only if
\begin{equation}\label{eq:lri}
    p=q=1\quad\text{or}\quad 1<p<\infty \text{ and } q\geq1\quad\text{or}\quad p=q=\infty.
\end{equation} 
In particular, when $p=q$, the Lorentz space $L^{p, p}(\R, \mu)$ coincides with the classical Lebesgue $L^p(\R, \mu)$ space.

Given two Banach spaces $X$ and $Y$ which embed into a common Hausdorff topological vector space, the \emph{$K$-functional} is defined as 
\begin{equation*}
    K(f, t, X, Y)=\inf_{f=f_0+f_1}\left(\|f_0\|_X+t\|f_1\|_Y\right).
\end{equation*}

The Holmstedt formulae (see \cite[Chapter 5, Theorem 1.9 and Theorem 2.1]{BS}) assert that if $1<p_0<q<\infty$ and $1\leq q_0\leq \infty$, then for $\frac{1}{\alpha}=\frac{1}{p_0}-\frac{1}{q}$ we have 
\begin{align*}
	K\left(f, t, L^{p_0, q_0}(\R, \mu), L^{q, 1}(\R, \mu)\right)&\approx \left(\int_0^{t^\alpha}s^{\frac{q_0}{p_0}-1}f^{*}(s)^{q_0}\d s\right)^\frac{1}{q_0}+t\int_{t^\alpha}^\infty s^{\frac{1}{q}-1}f^*(s)\d s
\intertext{for every $t>0$. Moreover, if $1<p_1<\infty$ and $1\leq q_1\leq\infty$, we have}
	K\left(f, t, L^{p_1, q_1}(\R, \mu), L^{\infty}(\R, \mu)\right)&\approx \left(\int_0^{t^{p_1}}s^{\frac{q_1}{p_1}-1}f^*(s)^{q_1}\d s\right)^\frac{1}{q_1}
\end{align*} with standard interpretations when either $q_0=\infty$ or $q_1=\infty$.

\section{Basic functional properties and examples}
Let us fix parameters for the remainder of the paper as in \eqref{eq:parameters} and, moreover, define $r=\frac{p_0q}{p_1(q-p_0)}$. In this section we provide a general treatise of operators exhibiting two end-point estimates 
\begin{equation}\label{eq:pq}
    T\colon L^{p_0, q_0}(\R, \mu)\to L^{p_1, q_1}(\S, \nu)\quad\text{and}\quad T\colon L^{q, 1}(\R, \mu)\to L^\infty(\S, \nu),
\end{equation} as well as the mapping properties of their associate Calderón operators.

Let us begin with the definition of the space $X^{\langle p_0, q_0\rangle}(\R, \mu)$ which will play a key role in the main theorem.
\begin{defn}
Let $p_0$ and $q_0$ be as in \eqref{eq:parameters} and let $X$ be a rearrangement-invariant space. We define $\|\cdot\|_{X^{\langle p_0, q_0\rangle}}\colon \M(\R, \mu)\to [0, \infty]$ by
\begin{align*}
    \|f\|_{X^{\langle p_0, q_0\rangle}}\begin{cases}=\left\|t^{-\frac{1}{p_0}}\left(\int_0^t s^{\frac{q_0}{p_0}-1}f^*(s)^{q_0}\d s\right)^\frac{1}{q_0}\right\|_{\bar X},\quad &q_0<\infty, \\
    =\left\|t^{-\frac{1}{p_0}}\sup_{0<s\leq t}s^\frac{1}{p_0}f^*(s)\right\|_{\bar X},\quad &q_0=\infty.\end{cases}
\end{align*} 
for every $f\in\M(\R, \mu)$. We then denote 
\begin{equation*}
    X^{\langle p_0, q_0\rangle}(\R, \mu)=\{f\in\M(\R, \mu):\|f\|_{X^{\langle p_0, q_0\rangle}}<\infty\}.
\end{equation*}
\end{defn}

Let us observe that for every $1\leq q_0\leq q_1\leq \infty$ the embeddings
\begin{equation}\label{eq:EmbeddingsXpq}
    X^{\langle p_0, q_0\rangle}(\R, \mu)\emb X^{\langle p_0, q_1\rangle}(\R, \mu)\emb X(\R, \mu)
\end{equation} 
hold with constants being irrespective of the underlying measure space. Indeed, the first embedding holds by virtue of the nesting property of Lorentz spaces, while the latter holds because the inequality
\begin{equation*}
    f^*(t)\leq t^{-\frac{1}{p_0}}\left(\int_0^t s^{\frac{q_1}{p_0}-1}f^*(s)^{q_1}\d s\right)^\frac{1}{q_1}
\end{equation*}
is true for every $f\in M^+(0, \infty), t>0$ and $q_1\in [1, \infty]$.

The following proposition deals with the question when $X^{\langle p_0, q_0\rangle}(\R, \mu)$ is a rearrangement-invariant space.

\begin{prop}\label{prop:Lorentzi}
Let $p_0$ and $q_0$ be as in \eqref{eq:parameters} and let $X(\R, \mu)$ be a rearrangement-invariant space. Then $X^{\langle p_0, q_0\rangle}(\R, \mu)$ is equivalent to a rearrangement-invariant space if and only if \begin{equation}\label{eq:nontriviality}
    \min\{1, t^{-\frac{1}{p_0}}\}\in \bar X.
\end{equation} In particular, when $\mu(\R)<\infty$, $X^{\langle p_0, q_0\rangle}(\R, \mu)$ is always a rearrangement-invariant space.
\end{prop}

\begin{proof}
To see that (P1) holds, recall that $p_0>1$, and so $\|f^*\|_{L^{p_0, q_0}(0, t)}\approx \|f^{**}\|_{L^{p_0, q_0}(0, t)}$ for every $t\in(0,\infty)$ with the constant being independent of $t>0$. Hence it follows that $\|f\|_{X^{\langle p_0, q_0\rangle}(\R, \mu)}\approx \|f^{**}\|_{\bar X^{\langle p_0, q_0\rangle}(\R, \mu)}, f\in\M(\R, \mu)$, the latter being a norm. Axioms (P2), (P3) and (P6) are seen to be satisfied. The axiom (P5) follows from \eqref{eq:EmbeddingsXpq}, since $X$ enjoys this property. Now, let $E\subset \R$ be a set with $t_0\coloneqq\mu(E)<\infty$. Since
\begin{equation*}
	\|\chi_{E}\|_{L^{p_0, q_0}(0, t)}=\left(\int_0^t s^{\frac{q_0}{p_0}-1}\chi_{(0, t_0)}(s)\d s\right)^\frac{1}{q_0}\approx \min\{t^\frac{1}{p_0}, t_0^\frac{1}{p_0}\},
\end{equation*} 
fixing $t_0$ yields that $\min\{t^\frac{1}{p_0}, t_0^\frac{1}{p_0}\}\approx \min\{t^\frac{1}{p_0}, 1\}, t>0$. Hence it follows that
\begin{equation*}
    \|\chi_E\|_{X^{\langle p_0, q_0\rangle}(\R, \mu)}<\infty\iff \left\|\min\{1, t^{-\frac{1}{p_0}}\}\right\|_{\bar X}<\infty.
\end{equation*}
This solves the case when $q_0<\infty$. The argument regarding the remaining case $q_0=\infty$ follows the same lines.
If $\mu(\R)<\infty$, then $\min\{1, t^{-\frac{1}{p_0}}\}\leq 1\in \bar X$, finishing the proof.
\end{proof}

\begin{rem}\label{rem:nontriviality}
    From the proof of Proposition~\ref{prop:Lorentzi}, it follows that the space $X^{\langle p_0, q_0\rangle}(\R, \mu)$ is rearrangement invariant if and only if it is nonempty which is characterized by \eqref{eq:nontriviality}.
\end{rem}

Let us note that the condition $\min\{1, t^{-\frac{1}{p_0}}\}\in \bar X$ is equivalent to inclusion $(L^{p_0, \infty}\cap L^\infty)(\R, \mu)\subset X(\R, \mu)$.

As for the inclusions in \eqref{eq:EmbeddingsXpq}, they may or may not be strict.
Indeed, let us first consider $L^\infty$. Then for every $t>0$ we have
\begin{equation}
    t^{-\frac{1}{p_0}}\int_0^t s^{\frac{1}{p_0}-1}f^*(s)\d s\leq \|f\|_\infty t^{-\frac{1}{p_0}}\int_0^t s^{\frac{1}{p_0}-1}\d s\approx \|f\|_\infty.
\end{equation}  
Taking the supremum over $t>0$ yields $L^\infty\emb (L^{\infty})^{\langle p_0, 1\rangle}$. Recalling \eqref{eq:EmbeddingsXpq}, we obtain $L^\infty=(L^\infty)^{\langle {p_0}, q_0\rangle}$ for every $q_0\in [1, \infty]$.

As for strict inclusions, consider $X=L^{p_0, \infty}$. Then, realizing that 
\begin{equation*}
    t\mapsto t^{-\frac{1}{p_0}}\left(\int_0^t s^{\frac{q_0}{p_0}-1} f^*(s)^{q_0}\d s\right)^\frac{1}{q_0} 
\end{equation*} 
is non-increasing, as a constant multiple of an integral mean of a non-increasing function with respect to the measure $s^{\frac{q_0}{p_0}-1}\d s$, we have 
\begin{equation}\label{eq:pq0}
    \sup_{t>0}t^\frac{1}{p_0}\left(\tau^{-\frac{1}{p_0}}\left(\int_0^\tau s^{\frac{q_0}{p_0}-1}f^*(s)^{q_0}\d s\right)^\frac{1}{q_0}\right)^*(t)= \sup_{t>0}\left(\int_0^t s^{\frac{q_0}{p_0}-1}f^*(s)^{q_0}\d s\right)^\frac{1}{q_0}=\|f\|_{p_0, q_0}.
\end{equation}

From now on, we will always assume that the spaces $X^{\langle p_0, q_0\rangle}(\R, \mu)$ are nontrivial, i.e.
\begin{equation*}
    (L^\infty\cap L^{p_0, \infty})(\R, \mu)\subset X(\R, \mu).
\end{equation*}

\begin{defn}
Let parameters be as in \eqref{eq:parameters}. We define operators $R_{q_0}$ and $S$ on $\M^+(0, \infty)$ by
\begin{align*}
R_{q_0}f(t)&= t^{-\frac{1}{p_1}}\left(\int_0^{t^r}s^{\frac{q_0}{p_0}-1}f^*(s)^{q_0}\d s\right)^\frac{1}{q_0}\\
\intertext{and}
Sf(t)&=\int_{t^r}^\infty s^{\frac{1}{q}-1}f^*(s)\d s.
\end{align*}
\end{defn}

\begin{prop}\label{prop:KINT}
Let parameters be as in \eqref{eq:parameters} and $T\colon \M(\R, \mu)\to \M(\S, \nu)$ be a quasilinear operator defined on $(L^{p_0, q_0}+L^{q, 1})(\R, \mu)$. Then the following statements are equivalent.
\begin{enumerate}[(i)]
	\item The operator $T$ satisfies \eqref{eq:tpq}.
	\item One has 
        \begin{equation*}
	           K(Tf, t, L^{p_1, q_1}, L^\infty)\lesssim K(f, t, L^{p_0, q_0}, L^{q, 1}),\quad f\in\M(\R, \mu), t>0.
	       \end{equation*}
	\item The inequality
        \begin{equation*}
	           t^{-\frac{1}{p_1}}\left(\int_0^{t}s^{\frac{q_1}{p_1}-1}\abs{Tf}^*(s)^{q_1}\d s\right)^\frac{1}{q_1}\lesssim R_{q_0}f^*(t)+Sf^*(t)\quad \text{holds for all $f\in \M(\R, \mu), t\in (0, \nu(\S))$.}
	       \end{equation*}
\end{enumerate}
\end{prop}
\begin{proof}
The implication $(i)\Rightarrow(ii)$ is an immediate consequence of the $K$-inequality (\cite[Chapter 5, Theorem 1.11]{BS}). Let us now show the converse implication. First, by the Holmstedt formula, we have
\begin{equation}\label{eq:KLPLinf}
    K(Tf, t, L^{p_1, q_1}, L^\infty)\approx\left(\int_0^{t^{p_1}}s^{\frac{q_1}{p_1}-1}\abs{Tf}^*(s)^{q_1}\d s\right)^\frac{1}{q_1},\quad f\in\M^+(\R, \mu), t>0.
\end{equation} 
Consequently,
\begin{equation*}
    \|Tf\|_{p_1, q_1}=\lim_{t\to\infty}\left(\int_0^{t^{p_1}}s^{\frac{q_1}{p_1}-1}\abs{Tf}^*(s)^{q_1}\d s\right)^\frac{1}{q_1}\lesssim \lim_{t\to\infty}K(f, t, L^{p_0, q_0}, L^{q, 1})\leq \|f\|_{p_0, q_0},\quad f\in\M^+(\R, \mu).
\end{equation*} 
Similarly,
\begin{align*}
    \|Tf\|_\infty=\lim_{t\to0^+}\abs{Tf}^*(t^{p_1})&\approx\lim_{t\to0^+}\frac{\abs{Tf}^*(t^{p_1})}{t}\left(\int_0^{t^{p_1}}s^{\frac{q_1}{p_1}-1}\d s\right)^\frac{1}{q_1}\\
	&\leq \lim_{t\to 0^+}\frac{1}{t}\left(\int_0^{t^{p_1}}s^{\frac{q_1}{p_1}-1}\abs{Tf}^*(s)^{q_1}\d s\right)^\frac{1}{q_1}\lesssim \lim_{t\to 0^+}\frac{K(f, t, L^{p_0, q_0}, L^{q, 1})}{t}\leq \|f\|_{q, 1}
\end{align*} 
for every $f\in\M^+(\R, \mu)$.
To show the implication $(ii)\Rightarrow(iii)$, we once again invoke Holmstedt formula and recall the definitions of operators $R_{q_0}$ and $S$ to obtain
\begin{equation*}
    K(f, t, L^{p, q_0}, L^{q, 1})\approx tR_{q_0}f^*(t^{p_1})+tSf^*(t^{p_1}),\quad f\in\M^+(\R, \mu), t>0.
\end{equation*} 
Putting this together with \eqref{eq:KLPLinf} yields
\begin{equation*}
    \left(\int_0^{t^{p_1}}s^{\frac{q_1}{p_1}-1}\abs{Tf}^*(s)^{q_1}\d s\right)^\frac{1}{q_1}\lesssim tR_{q_0}f^*(t^{p_1})+tSf^*(t^{p_1}),\quad f\in\M^+(\R, \mu), t>0.
\end{equation*} 
Dividing by $t$ and passing from $t$ to $t^\frac{1}{p_1}$ yields $(iii)$.

It thus remains to show that $(iii)$ implies $(ii)$ for $t>\nu(\S)^\frac{1}{p_1}$ when $\nu(\S)<\infty$, since the case $\nu(\S)=\infty$ was covered by the previous computation. Since $\nu(\S)<\infty$, we have $L^\infty(\S, \nu)\subset L^{p_1, q_1}(\S, \nu)$. Now, let $f\in(L^{p_1, q_1}+L^\infty)(\S, \nu)=L^{p_1, q_1}(\S, \nu)$ be given and let $f=f_0+f_1$ be any decomposition such that $f_0\in L^{p_1, q_1}(\S, \nu)$ and $f_1\in L^\infty(\S, \nu)$. Then for $t>\nu(\S)^{\frac{1}{p_1}}$ we have
\begin{equation*}
    K(f, t, L^{p_1, q_1}, L^\infty)\leq \|f\|_{p_1, q_1}\leq \|f_0\|_{p_1, q_1}+\|f_1\|_{p_1, q_1}\lesssim \|f_0\|_{p_1, q_1}+\nu(\S)^\frac{1}{p_1}\|f_1\|_\infty.
\end{equation*}
Consequently,
\begin{equation*}
	K(Tf, t, L^{p_1, q_1}, L^\infty)\lesssim K(Tf, \nu(\S)^{\frac{1}{p_1}}, L^{p_1, q_1}, L^\infty)\lesssim K(f, \nu(\S)^{\frac{1}{p_1}}, L^{p_0, q_0}, L^{q, 1})\leq K(f, t, L^{p_0, q_0}, L^{q, 1})
\end{equation*} 
for every $f\in \M(\R, \mu)$ and $t>\nu (\S)^\frac{1}{p_1}$.
\end{proof}

For the future reference, we now provide particular examples of spaces $X$ and $Y$ that admit the boundedness of operators $R_{q_0}$ and $S$ from the family of normable Lorentz spaces. This will be of particular interest for applications and, moreover, we will need a few special cases in upcoming results. 
As a by-product of the proof, we find an explicit form of the space $\left(L^{r_1, s_1}\right)^{\langle p_0, q_0\rangle}$, see Remark~\ref{rem:Lorentzlorentz} for precise statement. Moreover, for the remainder of the section, we shall write $L^{r_i, s_i}$ briefly for $L^{r_i, s_i}(0, \infty)$.

\begin{prop}\label{prop:Rq0}
Let parameters be as in \eqref{eq:parameters} and let $r_1, r_2, s_1, s_2$ be such that the Lorentz spaces $L^{r_i, s_i}, i=1, 2$, are equivalent to rearrangement-invariant spaces, see \eqref{eq:lri}. Then the operator $R_{q_0}$ is bounded from $L^{r_1, s_1}$ to $L^{r_2, s_2}$ if and only if 
\begin{equation}\label{eq:r1r2}
    \frac{1}{q}+\frac{1}{rr_2}=\frac{1}{r_1}
\end{equation} 
and one of the following conditions holds:
\begin{enumerate}[(i)]
    \item $r_1=p_0,\quad s_1\leq q_0\quad\text{and}\quad s_2=\infty$,
    \item $r\in(p_0, q]\quad\text{and}\quad 1\leq s_1\leq s_2\leq\infty$.
\end{enumerate}
\end{prop}

\begin{rem}
Let us briefly describe two pairs of parameters $r_1,r_2$ satisfying~\eqref{eq:r1r2}, which are in a certain sense extremal, namely
\begin{align*}
    r_1=p_0&\quad\text{and}\quad r_2=p_1\\
    \intertext{and}
    r_1=q&\quad\text{and}\quad r_2=\infty.
\end{align*} 
Additionally, since we intend to work only with normable Lorentz spaces, it is required to consider only $r_1\in \left[\max\left\{\frac{qp_0}{p_1(q-p_0)+p_0}, 1\right\}, q\right]$, on which we elaborate in Proposition~\ref{prop:BddS}. Moreover, note that when $r_1=q$, or equivalently $r_2=\infty$, one necessarily has $s_2=\infty$.
\end{rem}

\begin{proof}[Proof of Proposition~\ref{prop:Rq0}]
Let us first show that condition \eqref{eq:r1r2} is needed for the operator $R_{q_0}$ to be bounded. To this end, we denote $f_\lambda(t)=f(\lambda t)$ for $f\in\mathcal M(0, \infty)$. Then, changing variables, we get
\begin{equation}\label{eq:lr1}
    \|f_\lambda\|_{r_1, s_1}=\left\|t^{\frac{1}{r_1}-\frac{1}{s_1}}f_\lambda^*(t)\right\|_{s_1}\approx\left\|\lambda^{\frac{1}{s_1}-\frac{1}{r_1}-\frac{1}{s_1}}t^{\frac{1}{r_1}-\frac{1}{s_1}}f^*(t)\right\|_{s_1}=\lambda^{-\frac{1}{r_1}}\|f\|_{r_1, s_1}.
\end{equation}
Regarding the other term, changing variables, we similarly obtain
\begin{align*}
    \|R_{q_0}&f_\lambda(t)\|_{r_2, s_2}=\left\|t^{-\frac{1}{p_1}}\left(\int_0^{t^r}s^{\frac{q_0}{p_0}-1}f_\lambda^*(s)^{q_0}\d s\right)^\frac{1}{q_0}\right\|_{r_2, s_2}\approx\left\|\lambda^{-\frac{1}{p_0}}t^{-\frac{1}{p_1}}\left(\int_0^{\lambda t^r}s^{\frac{q_0}{p_0}-1}f^*(s)^{q_0}\d s\right)^\frac{1}{q_0}\right\|_{r_2, s_2}\\
    &=\left\|\lambda^{\frac{1}{rp_1}-\frac{1}{p_0}}\left(\lambda^{\frac{1}{r}}t\right)^{-\frac{1}{p_1}}\left(\int_0^{\left(\lambda^\frac{1}{r}t\right)^r}s^{\frac{q_0}{p_0}-1}f^*(s)^{q_0}\d s\right)^\frac{1}{q_0}\right\|_{r_2, s_2}=\lambda^{-\frac{1}{q}}\left\|R_{q_0}f(\lambda^\frac{1}{r}t)\right\|_{r_2, s_2}\approx \lambda^{-\frac{1}{q}-\frac{1}{rr_2}}\|R_{q_0}f\|_{r_2, s_2}.
\end{align*} 
Hence, putting these two equations together, we arrive exactly at condition \eqref{eq:r1r2}. The case $q_0=\infty$ follows the same lines.

Let us now characterize the boundedness of $R_{q_0}$ in the case $q_0<\infty$, the other case being once again similar.
Let us write 
\begin{equation}\label{eq:zmonot}
    R_{q_0}f(t)=t^{-\frac{1}{p_1}}\left(\int_0^{t^r}s^{\frac{q_0}{p_0}-1}f_\lambda^*(s)^{q_0}\d s\right)^\frac{1}{q_0}=t^{-\frac{1}{p_1}+\frac{r}{p_0}}\left(t^{-\frac{rq_0}{p_0}}\int_0^{t^r}s^{\frac{q_0}{p}-1}f_\lambda^*(s)^{q_0}\d s\right)^\frac{1}{q_0}
\end{equation} 
and observe that the quantity in parenthesis is non-increasing. Hence, by \cite[Lemma 3.1 (ii)]{GoPi2009} with $\beta=0$ and $\alpha=-\frac{1}{p_1}+\frac{r}{p_0}>0$ in their notation, we have 
\begin{align*}
    \int_0^x \sup_{y\geq t}R_{q_0}f(y)\d t\lesssim \int_0^x \left(R_{q_0}f\right)^*(t)\d t
\end{align*} 
for every $x\in (0, \infty)$ and $f\in\M(0, \infty)$. Consequently, by the Hardy-Littlewood-Pólya principle \eqref{eq:HLP}, we obtain
\begin{equation*}
    \|R_{q_0}f\|_{r_2, s_2}\approx\left\|\sup_{y\geq t}R_{q_0}f(y)\right\|_{r_2, s_2}=\left\|t^{\frac{1}{r_2}-\frac{1}{s_2}}\sup_{y\geq t} R_{q_0}f(y)\right\|_{s_2}.
\end{equation*} 
Now, if $s_2<\infty$, utilizing \eqref{eq:zmonot} once more, \cite[Theorem 3.2 (i)]{GoOpPi} tells us that
\begin{equation*}
    \left\|t^{\frac{1}{r_2}-\frac{1}{s_2}}\sup_{y\geq t} R_{q_0}f(y)\right\|_{s_2}\approx\left\|t^{\frac{1}{r_2}-\frac{1}{s_2}}R_{q_0}f(t)\right\|_{s_2}
\end{equation*} 
with weights
\begin{equation*}
    w(t)=t^{\frac{s_2}{r_2}-1},\quad u(t)=t^{\frac{r}{p_0}-\frac{1}{p_1}}\quad\text{and}\quad v(t)=t^{\frac{s_2}{r_2}+s_2\left(\frac{r}{p_0}-\frac{1}{p_1}\right)-1}
\end{equation*} 
in their notation. On the other hand, the case $s_2=\infty$ follows easily by exchanging the order of suprema. Now, changing variables and recalling \eqref{eq:r1r2}, we find that
\begin{equation}\label{eq:Rq0}
    \begin{split}
        \|R_{q_0}f\|_{r_2, s_2}&\approx\left\|t^{\frac{1}{r_2}-\frac{1}{s_2}+\frac{r}{p_0}-\frac{1}{p_1}}\left(t^{-\frac{rq_0}{p_0}}\int_0^{t^r}s^{\frac{q_0}{p_0}-1}f^*(s)^{q_0}\d s\right)^\frac{1}{q_0}\right\|_{s_2}\\
        &\approx \left\|t^{\frac{1}{rr_2}-\frac{1}{rs_2}+\frac{1}{p_0}-\frac{1}{rp_1}-\frac{1}{p_0}}\left(\int_0^t s^{\frac{q_0}{p}-1}f^*(s)^{q_0}\d s\right)^\frac{1}{q_0}t^{\frac{1}{rs_2}-\frac{1}{s_2}}\right\|_{s_2}\\
        &=\left\|t^{\frac{1}{r_1}-\frac{1}{s_2}}t^{-\frac{1}{p_0}}\left(\int_0^t s^{\frac{q_0}{p_0}-1}f^*(s)^{q_0}\d s\right)^\frac{1}{q_0}\right\|_{s_2}=\|f\|_{(L^{r_1, s_2})^{\langle p_0, q_0\rangle}}.
    \end{split}    
\end{equation}
Hence we have reduced the question of boundedness of $R_{q_0}$ into the question for which $s_1$ and $s_2$ the embedding
\begin{equation}
    L^{r_1, s_1}\emb \left(L^{r_1, s_2}\right)^{\langle p_0, q_0\rangle}
\end{equation} 
holds true. Here, we can immediately make few observations. In general, $s_1\leq s_2$ has to be satisfied. Indeed, this is an immediate consequence of embeddings of Lorentz spaces and \eqref{eq:EmbeddingsXpq}. Next notice that $r_1$ has to be greater than or equal to $p_0$. Indeed, if $r_1<p_0$, then $\left(L^{r_1, s_2}\right)^{\langle p_0, q_0\rangle}$ is empty by Remark~\ref{rem:nontriviality} and the embedding cannot hold. The same argument shows that if $r_1=p_0$, then necessarily $s_2=\infty$. Of course, we have already computed in \eqref{eq:pq0} that $\left(L^{p_0, \infty}\right)^{\langle p_0, q_0\rangle}=L^{p_0, q_0}$. Hence, when $r_1=p_0$, $s_1\leq q_0$ and $s_2=\infty$. 

We shall now show that when $r_1\in (p_0, q]$, then $\left(L^{r_1, s_2}\right)^{\langle p_0, q_0\rangle}=L^{r_1, s_2}$, which will conclude the proof. Additionally, we only need to show this equality for $q_0=1$, as then it will hold for every $q_0\in [1, \infty]$ by \eqref{eq:EmbeddingsXpq}.
We now distinguish two cases, based on whether $s_2$ is finite or infinite.

$s_2<\infty$: Using classical Hardy inequality, we estimate
\begin{align*}
    \|f\|_{\left(L^{r_1, s_2}\right)^{\langle p_0, 1\rangle}}&=\left(\int_0^\infty t^{\frac{s_2}{r_1}-1-\frac{s_2}{p_0}}\left(\int_0^t s^{\frac{1}{p_0}-1}f^*(s)\d s\right)^{s_2}\d t\right)^\frac{1}{s_2}\\
    &\lesssim\left(\int_0^\infty t^{\frac{s_2}{r_1}-1-\frac{s_2}{p_0}}t^{\frac{s_2}{p_0}-s_2}f^*(t)^{s_2}t^{s_2}\d t\right)^\frac{1}{s_2}=\left(\int_0^\infty t^{\frac{s_2}{r_1}-1}f^*(t)^{s_2}\right)^\frac{1}{s_2}=\|f\|_{r_1, s_2}.
\end{align*}

$s_2=\infty$: We estimate
\begin{align*}
    \|f\|_{\left(L^{r_1, \infty}\right)^{\langle p_0, 1\rangle}}=\sup_{t>0}t^{\frac{1}{r_1}-\frac{1}{p_0}}\int_0^t s^{\frac{1}{p_0}-1}f^*(s)\d s &=\sup_{t>0} t^{\frac{1}{r_1}-\frac{1}{p_0}+1}\left(s^{\frac{1}{p_0}-1}f^*(s)\right)^{**}(t)\\
    &\approx \sup_{t>0}t^{\frac{1}{r_1}-\frac{1}{p_0}+1+\frac{1}{p_0}-1}f^*(t)=\|f\|_{r_1, \infty},
\end{align*} 
where the last “$\approx$” holds true because $\frac{1}{r_1}-\frac{1}{p_0}<0$. This finishes the proof.
\end{proof}

\begin{rem}\label{rem:Lorentzlorentz}
    From the last part of the previous proof, it follows that whenever $r_1\in (p_0, \infty]$ and $s_1\in[1, \infty]$ are such that $L^{r_1, s_1}$ is a rearrangement-invariant function space, then $\left(L^{r_1, s_1}\right)^{\langle p_0, q_0\rangle}=L^{r_1, s_1}$ for every $q_0$ whenever $1\leq p_0<r_1\leq\infty$ 
\end{rem}
\begin{prop}\label{prop:BddS}
Let parameters be as in \eqref{eq:parameters} and let $r_1, r_2, s_1, s_2$ be such that the Lorentz spaces $L^{r_i, s_i}, i=1, 2$, are equivalent to rearrangement-invariant spaces. Then the operator $S$ is bounded from $L^{r_1, s_1}$ to $L^{r_2, s_2}$ if and only if \eqref{eq:r1r2} holds and one of the following conditions holds:

If $p_1\leq \frac{p_0(q-1)}{q-p_0}$:
\begin{enumerate}[(i)]
	\item $r_1=\frac{qp_0}{p_1(q-p_0)+p_0}\quad\text{and}\quad s_1=s_2=1$,
	\item $r_1\in (\frac{qp_0}{p_1(q-p_0)+p_0}, q)\quad\text{and}\quad 1\leq s_1\leq s_2\leq \infty$,
	\item $r_1=q\quad\text{and}\quad s_1=1$,
\end{enumerate}

else:
\begin{enumerate}[(i)]
    \item $r_1=1,\quad\text{and}\quad s_1=1\leq s_2\leq\infty$,
    \item $r_1\in (1, q)\quad\text{and}\quad 1\leq s_1\leq s_2\leq \infty$,
	\item $r_1=q\quad\text{and}\quad s_1=1$.
\end{enumerate}
\end{prop}

\begin{proof}
We prove the case $p_1\leq \frac{p_0(q-1)}{q-p_0}$, the latter case being similar. First, it can be shown that the validity of \eqref{eq:r1r2} is necessary for boundedness of $S$ similarly as it was shown for the operator $R_{q_0}$ in Proposition~\ref{prop:Rq0} via a scaling argument.  Condition $p_1\leq \frac{p_0(q-1)}{q-p_0}$ is equivalent to $r_1\coloneqq \frac{qp_0}{p_1(q-p_0)+p_0}\geq1$. Moreover, relation \eqref{eq:r1r2} maps this value of $r_1$ onto $r_2=1$. From the definition of the operator $S$ we immediately obtain equality $\|Sf\|_\infty=\|f\|_{q, 1}$. As for the other end-point estimate, utilizing Fubini's theorem, we obtain
\begin{align*}
	\|Sf\|_{1}&=\int_0^\infty \int_{t^r}^\infty s^{\frac{1}{q}-1}f^*(s)\d s\d t=\int_0^\infty\int_0^{s^\frac{1}{r}}s^{\frac{1}{q}-1}f^*(s)\d t\d s\\
    &=\int_0^\infty s^{\frac{1}{q}-1+\frac{1}{r}}f^*(s)\d s=\int_{0}^\infty s^{\frac{p_1(q-p_0)+p_0}{qp_0}-1}f^*(s)\d s=\|f\|_{\frac{qp_0}{p_1(q-p_0)+p_0}, 1}.
\end{align*}    
Hence, by Marcinkiewicz's interpolation theorem, $S$ is bounded from $L^{r_1, s_1}$ to $L^{r_2, s_2}$ whenever $r_1\in \left(\frac{qp_0}{p_1(q-p_0)+p_0}, q\right)$ and $1\leq s_1\leq s_2\leq \infty$. Now we just need to show that the parameters $s_1$ and $s_2$ cannot assume any other values. First, if $r_1=\frac{qp_0}{p_1(q-p_0)+p_0}$ then $r_2=1$ and so, necessarily $s_2=1$. Moreover, as $\|Sf\|_1=\|f\|_{\frac{qp_0}{p_1(q-p_0)+p_0}, 1}$, the situation where $s_1>1$ cannot occur due to the sharp nesting of Lorentz spaces. Similar argument reveals that when $r_1=q$, then it must be so that $s_1=1$ and $s_2=\infty$. Finally, if $r_1\in \left(\frac{qp_0}{p_1(q-p_0)+p_0}, q\right)$, let us assume $s_2<s_1$. A change of variables and relation \eqref{eq:r1r2} then show that
\begin{equation*}
    \|Sf\|_{r_2, s_2}^{s_2}\approx\int_0^\infty t^{\frac{s_2}{r_1}-\frac{s_2}{q}-1}\left(\int_t^\infty s^{\frac{1}{q}-1}f^*(s)\d s\right)^{s_2}\d t.
\end{equation*} 
Let us now consider the function $f_0$ defined by $f_0(t)=f_0^*(t)=\min\{1, t^{-\frac{1}{r_1}}\log(e-1+t)^{-\frac{1}{s_2}}\}$ for $t>0$. Then a brief computation shows that $f_0\in L^{r_1, s_1}$, but
\begin{align*}
    \|Sf_0\|_{r_2, s_2}^{s_2}\gtrsim \int_1^\infty t^{\frac{s_2}{r_1}-\frac{s_2}{q}-1}\left(\int_t^{2t} s^{\frac{1}{q}-1}t^{-\frac{1}{r_1}}\log(t)^{-\frac{1}{s_2}}\d s\right)^{s_2}\d t\approx \int_1^\infty t^{\frac{s_2}{r_1}-\frac{s_2}{q}-1}t^{\frac{s_2}{q}-\frac{s_2}{r_1}}\log(t)^{-1}\d t=\infty.
\end{align*}
\end{proof}

Combining Propositions~\ref{prop:Rq0} and \ref{prop:BddS} we can state the following result, characterizing boundedness of our Calderón operators in the setting of Lorentz spaces.

\begin{thm}
Let parameters be as in \eqref{eq:parameters} and let $r_1, r_2, s_1, s_2$ be such that the Lorentz spaces $L^{r_i, s_i}, i=1, 2$, are equivalent to rearrangement-invariant spaces and let $q_0\in [1, \infty]$. Then the operators $R_{q_0}$ and $S$ are bounded from $L^{r_1, s_1}$ to $L^{r_2, s_2}$ if and only if \eqref{eq:r1r2} holds and one of the following conditions holds:
\begin{enumerate}[(i)]
	\item $r_1=p_0,\quad s_1\leq q_0\quad\text{and}\quad s_2=\infty$,
	\item $r_1\in (p_0, q)\quad\text{and}\quad 1\leq s_1\leq s_2\leq \infty$,
	\item $r_1=q,\quad\text{and}\quad s_1=1$.
\end{enumerate}
\end{thm}

\begin{rem}
    It follows from Propositions~\ref{prop:Rq0} and \ref{prop:BddS} that the boundedness of the operator $R_{q_0}$ from $L^{r_1, s_1}$ to $L^{r_2, s_2}$ does not imply the boundedness of the operator $S$ in this situation and vice versa. It follows that neither Calderón suboperator can be omitted in the sense that, when wrapped in rearrangement-invariant norms, one term would bound the other.
\end{rem}

\section{Main results}
This section characterizes the boundedness of all linear operators satisfying \eqref{eq:tpq} via two one-dimensional Hardy-type operators.
We begin by stating and proving the Calderón-type theorem in the diagonal setting $q_0=q_1$.

\begin{thm}\label{thm:main}
Let parameters be as in \eqref{eq:parameters} with $q_0=q_1$ and let $X(\R, \mu)\subset(L^{p_0, q_0}+L^{q, 1})(\R, \mu)$ be a~rearrangement-invariant space. Then the following statements are equivalent.
\begin{enumerate}[(i)]
	\item Every linear operator satisfying \eqref{eq:tpq} is bounded from $X(\R, \mu)$ to $Y^{\langle p_1, q_0\rangle}(\S, \nu)$.
	\item Every quasilinear operator satisfying \eqref{eq:tpq} is bounded from $X(\R, \mu)$ to $Y^{\langle p_1, q_0\rangle}(\S, \nu)$.
	\item The operators $R_{q_0}$ and $S$ are bounded from $\bar X(0, \mu(\R))$ to $\bar Y(0, \nu(\S))$.
\end{enumerate}
\end{thm}

\begin{proof}
We begin with the implication $(iii)\Rightarrow(ii)$. Let $T$ be a quasilinear operator satisfying \eqref{eq:tpq}. By Proposition~\ref{prop:KINT} we have that
\begin{equation*}
    \begin{split}
        \|Tf\|_{Y^{\langle p_1, q_0\rangle}(\S, \nu)}&=\left\|t^{-\frac{1}{p_1}}\left(\int_0^{t}s^{\frac{q_0}{p_1}-1}\abs{Tf}^*(s)^{q_0}\d s\right)^\frac{1}{q_0}\right\|_{\bar Y(0, \nu(\S))}\\
        &\lesssim\|R_{q_0}f^*\|_{\bar Y(0, \nu(\S))}+\|Sf^*\|_{\bar Y(0, \nu(\S))}\lesssim\|f^*\|_{\bar X(0, \mu(\R))}=\|f\|_{X(\R, \mu)},\quad f\in\M^+(\R, \mu).
    \end{split}
\end{equation*} Since the implication $(ii)\Rightarrow(i)$ is trivial, it remains to prove $(i)\Rightarrow (iii)$. We begin with the operator $R_{q_0}$. We prove the case when $q_0<\infty$, as the proof of the remaining case $q_0=\infty$ is analogous. Using a change of variables, we have
\begin{equation}\label{eq:Rq012}
    \left\|R_{q_0}f\right\|_{\bar Y}=\left\|\frac{1}{t^\frac{1}{p_1}}\left(\int_0^{t^r} s^{\frac{q_0}{p_0}-1}f^*(s)^{q_0}\d s\right)^\frac{1}{q_0}\right\|_{\bar Y}\approx\left\|\frac{1}{t^\frac{1}{p_1}}\left(\int_0^t s^{\frac{q_0}{p_1}-1}[f^*(s^r)s^\frac{rp_1-p_0}{p_0p_1}]^{q_0}\d s\right)^\frac{1}{q_0}\right\|_{\bar Y}.
\end{equation}
We now define an operator on $\M^+(0, \infty)$ by
\begin{equation}\label{eq:recovery}
	\tilde Tf(t)=t^\frac{rp_1-p_0}{p_0p_1}f^{**}(t^r),\quad f\in\M^+(0, \infty), t>0.
\end{equation}
Let now $g=g^*\in\bar X(0, \mu(\R))$ be fixed. Without loss of generality, we may assume that $\supp g$ is bounded. Since $(\R, \mu)$ is nonatomic, we find $g_0\in\M^+(\R, \mu)$ be such that $g^*=g_0^*$. By virtue of \cite[Chapter 3, Theorem 2.13]{BS} we find $T_1\colon (L^1+L^\infty)(\R, \mu)\to (L^1+L^\infty)(0, \infty)$ satisfying $T_1f=\chi_{(0, \mu(\R))}T_1f$ for every $f\in(L^1+L^\infty)(\R, \mu)$, $\max\{\|T_1\|_{L^1(\R, \mu)\to L^1(0, \infty)}, \|T_1\|_{L^\infty(\R, \mu)\to L^\infty(0, \infty)}\}\leq 1$ and $T_1g_0=g^*$. Furthermore, the second property, by virtue of \cite[Chapter 3, Theorem 2.2]{BS}, implies that $\|T\|_{Z(\R, \mu)\to Z(0, \infty)}\leq 1$ for every rearrangement-invariant space $Z$.
We further define $$\tilde T_2f(t)=t^\frac{rp_1-p_0}{p_0p_1}\frac{1}{t^r}\int_0^{t^r}f(s)\chi_{(0, \mu(\R))}(s)\d s$$ and $T_2f(t)=\tilde T_2 f(t)\chi_{(0, \nu(S))}(t)$ for every $f\in\mathcal (L^1+L^\infty)(0, \infty)$ and $t>0$. Let us observe that $(\tilde T_2\circ T_1 f)(t)\leq t^\frac{rp_1-p_0}{p_0p_1}(T_1f)^{**}(t^r)$ for every $f\in \M(\R, \mu)$ and $t>0$. 
We now find a function $h\in\M^+(\S, \nu)$ satisfying $h^*=(T_2 g)^*$. Now, since the support of $g$ is bounded, $g\in L^1(0, \abs{\supp g})$. This implies that $h^*(t)\to0, t\to\infty$. Thus, by \cite[Chapter 2, Corollary 7.6]{BS} we may find a measure-preserving transformation $\sigma\colon \supp h\to \supp h^*$ such that $h=(h^*\circ\sigma)\chi_{\{\supp h\}}$ $\nu$-a.e. Setting $T_3f_0(x)=f_0(\sigma(x))\chi_{\{\supp h\}}$ for $f_0\in\M(0, \infty)$ and $x\in\S$, we have by \cite[Chapter 2, Proposition 7.2]{BS} that $(T_3f_0)^*=(f_0\chi_{\{\supp h\}})^*$ a.e. in $(0, \infty)$. 
Finally, we define the operator $T\colon \M(\R, \mu)\to \M(\S, \nu)$ by
\begin{equation}
T=T_3\circ T_2\circ T_1.
\end{equation} 
Observe that $T$ is a linear map, being a composition of linear maps. We show that $T$ satisfies \eqref{eq:tpq}. First, recalling the definition of $r$, we estimate
\begin{align*}
	\|Tf\|_{L^\infty(\S, \nu)}&=\|T_2\circ T_1 f\|_{L^\infty(0, \nu(\S))}\leq \|\tilde T_2 \circ T_1f\|_{L^\infty(0, \infty)}\leq\sup_{t>0}t^\frac{rp_1-p_0}{p_0p_1}(T_1f)^{**}(t^r)\\
	&=\sup_{t>0}t^\frac{rp_1-p_0}{rp_0p_1}(T_1f)^{**}(t)=\sup_{t>0}t^\frac{1}{q}(T_1f)^{**}(t)\approx \|T_1f\|_{L^{q,\infty}(0, \infty)}\lesssim\|T_1f\|_{L^{q, 1}(0, \infty)}\leq \|f\|_{L^{q,1}(\R, \mu)}
\end{align*} 
for every $f\in L^{q, 1}(\R, \mu)$.
To show that $T\colon L^{p_0, q_0}(\R, \mu)\to L^{p_1, q_0}(\S, \nu)$, it suffices to show that $\tilde T_2$ is bounded from $L^{p_0, q_0}(0, \infty)$ to $L^{p_1, q_0}(0, \infty)$. Indeed, once this is proved, we will get that
\begin{align*}
\|Tf\|_{L^{p_1, q_0}(\S, \nu)}=\|T_2\circ T_1f\|_{L^{p_1, q_0}(0, \nu(\S))}\leq \|\tilde T_2\circ T_1f\|_{L^{p_1, q_0}(0, \infty)}\lesssim \|T_1f\|_{L^{p_0, q_0}(0, \infty)}\leq\|f\|_{L^{p_0, q_0}(\R, \mu)}
\end{align*} 
for every $f\in L^{p_0, q_0}(\R, \mu)$.

Hence, let us prove the boundedness of $\tilde T_2$ from $L^{p_0, q_0}(0, \infty)$ to $L^{p_1, q_0}(0, \infty)$. We estimate
\begin{equation*}
    \|\tilde T_2f\|_{L^{p_1, q_0}(0, \infty)}\leq \left\|t^{\frac{1}{p_1}-\frac{1}{q_0}}\sup_{s\geq t}f^{**}(s^r)s^\frac{rp_1-p_0}{p_0p_1}\right\|_{L^{q_0}(0, \infty)}\lesssim \left\|t^{\frac{1}{p_1}-\frac{1}{q_0}}f^{**}(t^r)t^\frac{rp_1-p_0}{p_0p_1}\right\|_{L^{q_0}(0, \infty)}\approx \|f\|_{L^{p_0, q_0}(0, \infty)}. 
\end{equation*}
Here, the second inequality is \cite[Theorem 3.2 (i)]{BS} with weights
\begin{equation*}
    u(t)=t^\frac{rp_1-p_0}{p_0p_1},\quad v(t)=t^{\frac{q_0}{p_1}+q_0\frac{rp_1-p_0}{p_0p_1}-1}\quad\text{and}\quad w(t)=t^{\frac{q_0}{p_1}-1}
\end{equation*} 
in their notation. Finally, the last equivalence is a substitution $t^r\mapsto t$ combined with the boundedness of $f\mapsto f^{**}$ on $L^{p_0, q_0}(0, \infty)$.

Hence we conclude that $T$ satisfies \eqref{eq:tpq} and so $T\colon X(\R, \mu)\to Y^{\langle p_1, q_0\rangle}(\S, \nu)$ by assumption.

Coming back to \eqref{eq:Rq012} we now aim to show that 
\begin{equation}\label{eq:Rq1}
    \begin{split}
        &\left\|\frac{1}{t^\frac{1}{p_1}}\left(\int_0^t s^{\frac{q_0}{p_1}-1}[f^*(s^r)s^\frac{rp_1-p_0}{p_0p_1}]^{q_0}\d s\right)^\frac{1}{q_0}\right\|_{\bar Y}\\
        \lesssim &\left\|\frac{1}{t^\frac{1}{p_1}}\left(\int_0^t s^{\frac{q_0}{p_1}-1}[f^*(u^r)u^\frac{rp_1-p_0}{p_0p_1}]^*(s)^{q_0}\d s\right)^\frac{1}{q_0}\right\|_{\bar Y},\quad f\in\M^+(0, \infty).
    \end{split}
\end{equation} 
Now, by \cite[Lemma 3.1 (ii)]{GoPi2009}, we have that 
\begin{equation}\label{eq:supHLP}
    \int_0^t\sup_{s\leq y}y^\frac{rp_1-p_0}{p_0p_1}h(y)\d s\lesssim \int_0^t [x^\frac{rp_1-p_0}{p_0p_1}h(x)]^*(s)\d s
\end{equation} 
for every $h\geq 0$ non-increasing with $\beta=0$ and $\alpha =\frac{rp_1-p_0}{p_0p_1}>0$ in their notation, in which the constants in `$\lesssim$' do not depend neither on $t$ nor on $h$. Hence the Hardy-Littlewood-Pólya principle \eqref{eq:HLP} yields
\begin{equation*}
    \begin{split}
            \left(\int_0^t s^{\frac{q_0}{p_1}-1}[f^*(s^r)s^\frac{rp_1-p_0}{p_0p_1}]^{q_0}\d s\right)^\frac{1}{q_0}&\leq \left(\int_0^t s^{\frac{q_0}{p_1}-1}[\sup_{y\geq s}f^*(y^r)y^\frac{rp_1-p_0}{p_0p_1}]^{q_0}\d s\right)^\frac{1}{q_0}\\
            &\lesssim  \left(\int_0^t s^{\frac{q_0}{p_1}-1}[f^*(x^r)x^\frac{rp_1-p_0}{p_0p_1}]^*(s)^{q_0}\d s\right)^\frac{1}{q_0}
    \end{split}
\end{equation*}
which implies \eqref{eq:Rq1}. Finally, using a change of variables, \eqref{eq:Rq1} and definitions of operators $T_i$, we estimate
\begin{equation*}
    \begin{split}
        \|Rq_0g\|_{\bar Y(0, \nu(\S)}&\approx\left\|\frac{1}{t^\frac{1}{p_1}}\left(\int_0^t s^{\frac{q_0}{p_1}-1}[g^*(s^r)s^\frac{rp_1-p_0}{p_0p_1}]^{q_0}\d s\right)^\frac{1}{q_0}\right\|_{\bar Y(0, \nu(\S))}\\
        &=\left\|\frac{1}{t^\frac{1}{p_1}}\left(\int_0^t s^{\frac{q_0}{p_1}-1}[\chi_{(0, \nu(\S))}(s)T_1g_0(s^r)s^\frac{rp_1-p_0}{p_0p_1}]^{q_0}\d s\right)^\frac{1}{q_0}\right\|_{\bar Y(0, \nu(\S))}\\
    	&\lesssim \left\|\frac{1}{t^\frac{1}{p_1}}\left(\int_0^t s^{\frac{q_0}{p_1}-1}[\chi_{(0, \nu(\S))}(u)(T_1g_0)(u^r)u^\frac{rp_1-p_0}{p_0p_1}]^*(s)^{q_0}\d s\right)^\frac{1}{q_0}\right\|_{\bar Y(0, \nu(\S))}\\
    	&\leq\left\|\frac{1}{t^\frac{1}{p_1}}\left(\int_0^t s^{\frac{q_0}{p_1}-1}[(T_2\circ T_1g_0)(u)]^*(s)^{q_0}\d s\right)^\frac{1}{q_0}\right\|_{\bar Y(0, \nu(\S))}\\
    	&\leq \left\|\frac{1}{t^\frac{1}{p_1}}\left(\int_0^t s^{\frac{q_0}{p_1}-1}[(Tg_0)(u)]^*(s)^{q_0}\d s\right)^\frac{1}{q_0}\right\|_{\bar Y(0, \nu(\S))}\\
    	&=\|Tg_0\|_{Y^{\langle p_1, q_0\rangle}(\S, \nu)}\lesssim \|g_0\|_{X(\R, \mu)}=\|g\|_{\bar X(0, \mu(\R))}.
    \end{split}
\end{equation*} 
In other words, $R_{q_0}$ is bounded from $\bar X(0, \mu(\R))$ to $\bar Y(0, \nu(\S))$.

Finally, we only need to handle the operator $S$. Here, we first define an auxiliary operator $S_1$ by 
\begin{equation*}
    S_1f(t)=\int_{t^r}^\infty s^{\frac{1}{q}-1}f(s)\d s,\quad f\in\M(0, \infty), t>0.
\end{equation*} 
Using the same arguments as in Proposition~\ref{prop:BddS} for operator $S_1$ in place of $S$, we get $S_1\colon L^{p_0, q_0}(0, \infty)\to L^{p_1, q_0}(0, \infty)$ and $S_1\colon L^{q, 1}(0, \infty)\to L^\infty(0, \infty)$. Obviously $S_1f^*=Sf$ for every $f\in\M(0, \infty)$ and so it suffices to establish the result for operator $S_1$. To this end, we only need to construct the appropriate transformations between the measure spaces so that $S_1$ satisfies \eqref{eq:tpq}.

Let $g=g^*\in \bar X(0, \mu(\R))$ be given. Once again we find $g_0\in\M^+(\R, \mu)$ and $T_1\colon (L^{1}+L^{\infty})(\R, \mu)\to T\colon (L^{1}+L^{\infty})(0, \infty)$ satisfying $g_0^*=g^*$, $\max\{\|T_1\|_{L^1\to L^1}, \|T_1\|_{L^\infty\to L^\infty}\}\leq 1$, $T_1g_0=g^*$and $T_1 f=\chi_{(0, \mu(\R))}T_1f$ for every $f\in\M^+(\R, \mu)$. Next, we find a function $h\in (\S, \nu)$ such that $h^*=(S_1\circ T_1f)\chi_{(0, \nu(\S))}$. Let us now observe that $S_1f(t)\to 0$ as $t$ approaches infinity for every $f\in(L^{p_0, q_0}+L^{q, 1})(0, \infty)$. For functions from $L^{q, 1}(0, \infty)$, this is nothing else than the dominated convergence theorem. If $f\in L^{p_0, q_0}(0, \infty)$, then we know that $S_1f\in L^{p_1, q_0}(0, \infty)$. 
Since $S_1f$ is non-increasing and $p_1<\infty$, necessarily $S_1f(t)\to 0, t\to\infty$. Hence, there exists a measure-preserving transformation $\sigma$ from $\supp h$ onto $\supp h^*$ such that $h=(h^*\circ \sigma)\chi_{\rm{supp} h}$ where $h(t)=S_1g^*(t)$. Set $T_2f(x)=f(\sigma(x))\chi_{\supp h}$ for $f\in\mathcal M(0, \infty)$ and $x\in\S$. Then we have $(T_2 f_0)^*= (f_0\chi_{\supp h})^*$ a.e. in $(0, \infty)$ for every $f_0\in\M^+(0, \infty)$.  Finally, we define an operator $\tilde S\colon \mathcal M(\R, \mu)\to \mathcal M(\S, \nu)$ by $\tilde Sg(x)=(T_2\circ S_1\circ T_1 g)(x)$ for $g\in\mathcal M(\R, \mu)$ and $x\in\S$. The operator $\tilde S$ is linear by construction. Next, we claim that $\tilde S$ satisfies \eqref{eq:tpq}. To show this, we verify its boundedness between appropriate spaces. We show boundedness from $L^{q, 1}(\R, \mu)$ to $L^{\infty}(\S, \nu)$. We estimate
\begin{equation}
    \|\tilde S f\|_{L^\infty(\S, \nu)}=\|T_2\circ S_1\circ T_1f\|_{L^\infty(\S, \nu)}\leq \|S\circ T_1f\|_{L^\infty(0, \infty)}\lesssim \|T_1f\|_{L^{q, 1}(0, \infty)}= \|f\|_{L^{q, 1}(\R, \mu)}
\end{equation} 
for every $f\in L^{q, 1}(\R, \mu)$. The proof of the boundedness from $L^{p_0, q_0}(\R, \mu)$ to $L^{p_1, q_0}(\S, \nu)$ is analogous. Hence $\tilde S\colon X(\R, \mu)\to Y^{\langle p_1, q_0\rangle}(\S, \nu)$. Finally, as $Y^{\langle p_1, q_0 \rangle}(\S, \nu)\emb Y(\S, \nu)$ and using the fact that $(T_2\circ S_1\circ T_1 g_0)^*=S_1g=Sg$ a.e. in $(0, \nu(\S))$, we obtain that 
\begin{equation*}
    \|Sg\|_{\bar Y(\S, \nu)}=\|T_2\circ S_1\circ T_1 g_0\|_{Y(\S, \nu)}=\|\tilde Sg_0\|_{Y(\S, \nu)}\leq \|\tilde Sg_0\|_{Y^{\langle p_1, q_0 \rangle}(\S, \nu)}\lesssim \|g_0\|_{X(\R, \mu)}=\|g\|_{\bar X(0, \mu(\R))}.
\end{equation*}
\end{proof}

What follows is a corollary of extrapolation nature which essentially tells us that the boundedness of operators from $L^{p_0, q_0}(\R, \mu)$ to $L^{p_1, q_0}(\S, \nu)$ is a stronger property than the boundedness of operators from $L^{p_0, q_1}(\R, \mu)$ to $L^{p_1, q_1}(\S, \nu)$ for $q_0\leq q_1$.

\begin{cor}\label{cor:q0q1}
Let parameters be as in \eqref{eq:parameters}. Let $X(\R, \mu)$ be a rearrangement-invariant space satisfying $X(\R, \mu)\subset (L^{p_0, q_0}+L^{q, 1})(\R, \mu)$. Consider the following statements.
\begin{enumerate}[(i)]
	\item Every linear operator satisfying 
		\begin{equation*}
			T\colon L^{p_0, q_0}(\R, \mu)\to L^{p_1, q_0}(\S, \nu)\quad\text{and}\quad T\colon L^{q, 1}(\R, \mu)\to L^\infty(\S, \nu)
		\end{equation*}
		is bounded from $X(\R, \mu)$ to $Y^{\langle p_1, q_0\rangle}(\S, \nu)$.
	\item Every linear operator satisfying 
		\begin{equation*}
			T\colon L^{p_0, q_1}(\R, \mu)\to L^{p_1, q_1}(\S, \nu)\quad\text{and}\quad T\colon L^{q, 1}(\R, \mu)\to L^\infty(\S, \nu)
		\end{equation*}
		is bounded from $X(\R, \mu)$ to $Y^{\langle p_1, q_1\rangle}(\S, \nu)$.
\end{enumerate} Then $(i)$ implies $(ii)$.
\end{cor}

\begin{proof}
By Theorem~\ref{thm:main} statement $(i)$ (resp. $(ii)$) is equivalent to the boundedness of operators $S$ and $R_{q_0}$ (resp. $S$ and $R_{q_1}$) from $\bar X(0, \mu(\R))$ to $\bar Y(0, \nu(\S))$. Hence, it suffices to show that the boundedness of $R_{q_0}$ implies that of $R_{q_1}$. We however, by virtue of nesting properties of Lorentz spaces, know that
\begin{equation*}
    R_{q_1}f(t)\lesssim R_{q_0}f(t),\quad f\in\M^+(0, \infty), t>0.
\end{equation*} 
This and our assumption imply 
\begin{equation*}
    \|R_{q_1}f\|_{\bar Y}\lesssim\|R_{q_0}f\|_{\bar Y}\lesssim \|f\|_{\bar X},\quad f\in\M^+(0, \infty),
\end{equation*} 
whence the results follows.
\end{proof}

Let us now discuss whether condition $(i)$ is strictly stronger than condition $(ii)$ in Corollary~\ref{cor:q0q1} with $1\leq q_0<q_1\leq\infty$. To this end, let us define $\|f\|_{Y_{q_i}}=\|R_{q_i}f\|_{Y}$ for $i=0, 1$. Then $Y_{q_i}$ is a linear space and $Y_{q_0}$ is nonempty if and only if $Y_{q_1}$ is. It was shown in Proposition~\ref{prop:Rq0} above that when $Y$ is the Lorentz space $L^{r_2, s_2}$, then $Y_{q_i}=\left(L^{r_1, s_2}\right)^{\langle p_0, q_i\rangle}$, where $r_1$ and $r_2$ satisfy \eqref{eq:r1r2}. On taking $r_1=p_0$ and $s_2=\infty$, we hence have that $Y_{q_i}=L^{p_0, q_i}$. Taking $X=L^{p_0, q_2}$, where $q_2=\frac{1}{2}(q_0+q_1)$, we have that $R_{q_i}$ si bounded from $X$ to $Y$ if and only if $i=1$. Here $(L^{p_0, q_0}+L^{q, 1})$ does not contain $X$ but its intersection with $X$ forms a dense subspace of $X$. Similarly, it may happen that $Y_{q_0}=Y_{q_1}$ and so the conditions $(i)$ and $(ii)$ are equivalent.

We will now consider the case $q_0<q_1$. The need to separate this case from the `diagonal' case $q_0=q_1$ is mainly due to the problems that arise when we try to achieve an inequality similar to \eqref{eq:Rq012}, but with a~parameter $q_1$ in place of $q_0$ in the right-most expression. For instance, such inequality is impossible to achieve pointwise with the operator $\Tilde T$ defined in \eqref{eq:recovery}. In order to treat the case $q_0<q_1$, we will need to define a new non-linear operator $Y_{q_0, q_1}$ which we will see is closely related to some operators we have already encountered, see Remark~\ref{rem:Yq0q1}.

\begin{defn}
Let parameters be as in \eqref{eq:parameters} with $1\leq q_0< q_1\leq\infty$. We define operator $Y_{q_0, q_1}f$ by
\begin{equation}
    Y_{q_0, q_1}f(t)=\left(\int_0^t s^{\frac{rq_0}{p_0}-1}f^*(s^r)^{q_0}\d s\right)^\frac{q_1-q_0}{q_0q_1}t^{\frac{rq_0}{p_0q_1}-\frac{1}{p_1}}f^*(t^r)^\frac{q_0}{q_1},\quad f\in\M^+(0, \infty), t>0.
\end{equation}
\end{defn}

\begin{rem}\label{rem:Yq0q1}
Observe that if we formally set $q_0=q_1$, we have $$Y_{q_0, q_1}f(t)=t^{\frac{rp_1-p_0}{p_0p_1}}f^*(t^r).$$
In other words, we recover the operator which played a crucial role in the proof of Theorem~\ref{thm:main}. Moreover, if $q_1=\infty$, then $$Y_{q_0, \infty}f(t)=R_{q_0}f(t).$$
\end{rem}

The following proposition establishes that the operator $Y_{q_0, q_1}$ satisfies \eqref{eq:tpq}, and therefore does not deviate from our built framework.
\begin{prop}\label{prop:Y}
Let parameters be as in \eqref{eq:parameters} with $1\leq q_0<q_1\leq\infty$. Then $Y=Y_{q_0, q_1}$ is bounded from $L^{p_0, q_0}(0, \infty)$ and $L^{q, 1}(0, \infty)$ to $L^{p_1, q_1}(0, \infty)$ and $L^\infty(0, \infty)$ respectively.
\end{prop}
\begin{proof}
Let us begin with the case when $q_1<\infty$. We first check that $Y\colon L^{q, 1}(0, \infty)\to L^\infty(0, \infty)$. We have
\begin{align*}
	\sup_{t>0}Yf(t)&=\sup_{t>0}\left(\int_0^t s^{\frac{rq_0}{p_0}-1}f^*(s^r)^{q_0}\d s\right)^\frac{q_1-q_0}{q_0q_1}t^{\frac{rq_0}{p_0q_1}-\frac{1}{p_1}}f^*(t^r)^\frac{q_0}{q_1}\\
	&\approx \sup_{t>0}\left(\int_0^t s^{\frac{q_0}{p_0}-1}f^*(s)^{q_0}\d s\right)^\frac{q_1-q_0}{q_0q_1}t^{\frac{q_0}{p_0q_1}-\frac{1}{rp_1}}f^*(t)^\frac{q_0}{q_1}\\
	&\approx\sup_{t>0}\left(\int_0^t s^{\frac{q_0}{p_0}-1}f^*(s)^{q_0}\d s\right)^\frac{q_1-q_0}{q_0q_1}t^{-\frac{1}{rp_1}}\left(\int_0^t s^{\frac{1}{p_0}-1}\d s\right)^\frac{q_0}{q_1}f^*(t)^\frac{q_0}{q_1}\\
	&\leq \sup_{t>0} \|f\|_{L^{p_0, q_0}(0, t)}^\frac{q_1-q_0}{q_1}t^{-\frac{1}{rp_1}}\|f\|_{L^{p_0, 1}(0, t)}^{\frac{q_0}{q_1}}\lesssim \sup_{t>0}\|f\|_{L^{p_0, 1}(0, t)}t^{-\frac{1}{rp_1}}\\
	&=\sup_{t>0} t^{-\frac{1}{rp_1}}\int_0^t t^{\frac{1}{p_0}-1}f^*(s)\d s\leq\sup_{t>0} \int_0^t s^{\frac{1}{p_0}-\frac{1}{rp_1}-1}f^*(s)\d s=\sup_{t>0}\|f\|_{L^{q, 1}(0, t)}=\|f\|_{L^{q, 1}(0, \infty)},
\end{align*}
where all the estimates are independent of both $f$ and $t$. Next, we show that $Y\colon L^{p_0, q_0}\to L^{p_1, q_1}$. We estimate
\begin{align*}
	\|Yf\|_{L^{p_1, q_1}(0, \infty)}&=\left(\int_0^\infty t^{\frac{q_1}{p_1}-1}(Yf)^*(t)^{q_1}\d t\right)^\frac{1}{q_1}\\
	&\leq \left(\int_0^\infty t^{\frac{q_1}{p_1}-1}\sup_{y\geq t}\left(\int_0^y s^{\frac{rq_0}{p_0}-1}f^*(s^r)^{q_0}\d s\right)^\frac{q_1-q_0}{q_0}y^{\frac{rq_0}{p_0}-\frac{q_1}{p_1}}f^*(y^r)^{q_0}\d t\right)^\frac{1}{q_1}\\
	&\leq \|f\|_{L^{p_0, q_0}(0, \infty)}^\frac{q_1-q_0}{q_1}\left(\int_0^\infty t^{\frac{q_1}{p_1}-1}\sup_{y\geq t} y^{\frac{rq_0}{p_0}-\frac{q_1}{p_1}}f^*(y^r)^{q_0}\d t\right)^\frac{1}{q_1}\\
	&\lesssim \|f\|_{L^{p_0, q_0}(0, \infty)}^\frac{q_1-q_0}{q_1}\left(\int_0^\infty t^{\frac{q_1}{p_1}-1}t^{\frac{rq_0}{p_0}-\frac{q_1}{p_1}}f^*(t^r)^{q_0}\d t\right)^\frac{1}{q_1}\\
	&\approx \|f\|_{L^{p_0, q_0}(0, \infty)}^\frac{q_1-q_0}{q_1}\|f\|_{L^{p_0, q_0}(0, \infty)}^\frac{q_0}{q_1}=\|f\|_{L^{p_0, q_0}(0, \infty)},
\end{align*} 
where the penultimate line comes from \cite[Lemma 3.1.(i)]{GoPi2009}. For the remaining case $q_1=\infty$, by Remark~\ref{rem:Yq0q1} we have $Y_{q_0, \infty}f=R_{q_0}f$ for every $f\in \M(0, \infty)$. Thus, the result follows from Proposition~\ref{prop:Rq0}. 
\end{proof}

We now state a result which is analogous to Theorem~\ref{thm:main} in the case $q_1=\infty$ with the main exception that $q_0<\infty$.

\begin{thm}\label{thm:udia}
Let parameters be as in \eqref{eq:parameters} with $q_0<q_1=\infty$ and $X(\R, \mu)\subset(L^{p_0, q_0}+L^{q, 1})(\R, \mu)$ be a~rearrangement-invariant space. Then the following statements are equivalent.
\begin{enumerate}[(i)]
	\item Every sublinear operator satisfying \eqref{eq:tpq} is bounded from $X(\R, \mu)$ to $Y^{\langle p_1, \infty\rangle}(\S, \nu)$.
	\item Every quasilinear operator satisfying \eqref{eq:tpq} is bounded from $X(\R, \mu)$ to $Y^{\langle p_1, \infty\rangle}(\S, \nu)$.
	\item The operators $R_{q_0}$ and $S$ are bounded from $\bar X(0, \mu(\R))$ to $\bar Y(0, \nu(\S))$.
\end{enumerate}
\end{thm}
\begin{proof}
    By Proposition~\ref{prop:Y} the~operator $R_{q_0}$ is bounded from $L^{p_0, q_0}(0, \infty)$ and $L^{q, 1}(0, \infty)$ to $L^{p_1, \infty}(0, \infty)$ and $L^\infty(0, \infty)$, respectively, and is equivalent to a sublinear operator (by passing from $f^*$ to $f^{**}$ in its definition). Therefore, it suffices to find appropriate transformations between measure spaces so that “$R_{q_0}$” maps from $\mathcal M^+(\R, \mu)$ to $\mathcal M^+(\S, \nu)$. This can be done in a similar fashion as in Theorem~\ref{thm:main} and thus the proof is omitted.
\end{proof}

It is not difficult to see that in the case $q_0=1$, the operator $R_1$ can be linearized, as we need not take roots. Hence we can state the following corollary.

\begin{cor}
Let parameters be as in \eqref{eq:parameters} with $q_0=1$ and $q_1=\infty$. Let $X(\R, \mu)\subset(L^{p_0, 1}+L^{q, 1})(\R, \mu)$ be a~rearrangement-invariant space. Then the following statements are equivalent.
        \begin{enumerate}[(i)]
        \item Every linear operator satisfying \eqref{eq:tpq} is bounded from $X(\R, \mu)$ to $Y^{\langle p_1, \infty\rangle}(\S, \nu)$.
         \item Every quasilinear operator satisfying \eqref{eq:tpq} is bounded from $X(\R, \mu)$ to $Y^{\langle p_1, \infty\rangle}(\S, \nu)$.
        \item The operators $R_{1}$ and $S$ are bounded from $\bar X(0, \mu(\R))$ to $\bar Y(0, \nu(\S))$.
    \end{enumerate}
\end{cor}

The following theorem deals with the remaining case where $q_1<\infty$. Here, we are unable to recover a~complete analogue of Theorem~\ref{thm:main}. However, we are still able to recover the boundedness of the Calderón operator through the operator $Y_{q_0, q_1}$ which, by Proposition~\ref{prop:Y}, satisfies \eqref{eq:tpq}.
\begin{thm}
Let parameters be as in \eqref{eq:parameters} with $1\leq q_0<q_1<\infty$ and $X(\R, \mu)\subset (L^{p_0, q_0}+L^{q, 1})(\R, \mu)$ be a~rearrangement-invariant space. Let $Y(\S, \nu)$ be a~rearrangement-invariant space. Assume that $Y_{q_0, q_1}$ is bounded from $\bar X(0, \mu(\R))$ to $\bar Y^{\langle p_1, q_1\rangle}(0, \nu(\S))$. Then $R_{q_0}$ is bounded from $\bar X(0, \mu(\R))$ to $\bar Y(0, \nu(\S))$.
\end{thm}
\begin{proof}
We estimate
\begin{equation}\label{eq:4.2.}
    \begin{split}
        \|R_{q_0}f\|_{\bar Y(0, \nu(\S))}&\approx\left\|t^{-\frac{1}{p_1}}\left(\int_0^t s^{\frac{rq_0}{p_0}-1}f^*(s^r)^{q_0}\d s\right)^\frac{1}{q_0}\right\|_{\bar Y(0, \nu(\S))}=\left\|t^{-\frac{1}{p_1}}\left(\left(\int_0^t s^{\frac{rq_0}{p_0}-1}f^*(s^r)^{q_0}\d s\right)^\frac{q_1}{q_0}\right)^\frac{1}{q_1}\right\|_{\bar Y(0, \nu(\S))}\\
    	&\approx \left\|t^{-\frac{1}{p_1}}\left(\int_0^t \left(\int_0^s \tau^{\frac{rq_0}{p_0}-1}f^*(\tau^r)^{q_0}\d \tau\right)^\frac{q_1-q_0}{q_0}s^{\frac{rq_0}{p_0}-1}f^*(s^r)^{q_0}\d s\right)^\frac{1}{q_1}\right\|_{\bar Y(0, \nu(\S))}\\
    	&=\left\|t^{-\frac{1}{p_1}}\left(\int_0^t s^{\frac{q_1}{p_1}-1}Yf(s)^{q_1}\d s\right)^\frac{1}{q_1}\right\|_{\bar Y(0, \nu(\S))}\leq \left\|t^{-\frac{1}{p_1}}\left(\int_0^t s^{\frac{q_1}{p_1}-1}\sup_{y\geq s}Yf(y)^{q_1}\d s\right)^\frac{1}{q_1}\right\|_{\bar Y(0, \nu(\S))}.
    \end{split}
\end{equation} 
We now claim that 
\begin{equation}\label{eq:claim1}
    \int_0^t s^{\frac{q_1}{p_1}-1}\sup_{y\geq s}Yf(y)^{q_1}\d s\lesssim \int_0^t s^{\frac{q_1}{p_1}-1}(Yf)^*(s)^{q_1}\d s,\quad f\in\M(0, \infty), t>0.
\end{equation}
Indeed, we have
\begin{align*}
	\int_0^t s^{\frac{q_1}{p_1}-1}&\sup_{y\geq s}Yf(y)^{q_1}\d s\leq\int_0^t s^{\frac{q_1}{p}-1}\sup_{y\geq s}\left(\int_0^y s^{\frac{rq_0}{p_0}-1}f^*(s^r)^{q_0}\d s\right)^\frac{q_1-q_0}{q_0}y^{\frac{rq_0}{p_0}-\frac{q_1}{p_1}}f^*(y^r)^{q_0}\d s\\
	&\approx \int_0^{t^{q_1}} s^{\frac{1}{p_1}-1}\sup_{y\geq s^\frac{1}{q_1}}\left(\int_0^y s^{\frac{rq_0}{p_0}-1}f^*(s^r)^{q_0}\d s\right)^\frac{q_1-q_0}{q_0}y^{\frac{rq_0}{p_0}-\frac{q_1}{p_1}}f^*(y^r)^{q_0}\d s\\
	&\approx\int_0^{t^{q_1}} s^{\frac{1}{p_1}-1}\sup_{y\geq s^\frac{1}{q_1}}\left(\int_0^{y^\frac{rq_0}{p_0}} f^*(\tau^\frac{p_0}{q_0})^{q_0}\d \tau\right)^\frac{q_1-q_0}{q_0}y^{\frac{rq_0}{p_0}-\frac{q_1}{p_1}}f^*(y^r)^{q_0}\d s\\
	&=\int_0^{t^{q_1}} s^{\frac{1}{p_1}-1}\sup_{y\geq s}\left(\int_0^{y^\frac{rq_0}{p_0q_1}} f^*(\tau^\frac{p_0}{q_0})^{q_0}\d \tau\right)^\frac{q_1-q_0}{q_0}y^{\frac{rq_0}{q_1p_0}-\frac{1}{p_1}}f^*(y^\frac{r}{q_1})^{q_0}\d s\\
	&=\int_0^{t^{q_1}} s^{\frac{1}{p_1}-1}\sup_{y\geq s}\left(y^{-\frac{rq_0}{p_0q_1}}\int_0^{y^\frac{rq_0}{p_0q_1}} f^*(\tau^\frac{p_0}{q_0})^{q_0}\d \tau\right)^\frac{q_1-q_0}{q_0}y^{\frac{rq_0}{p_0q_1}\cdot\frac{q_1-q_0}{q_0}}\cdot y^{\frac{rq_0}{q_1p_0}-\frac{1}{p_1}}f^*(y^\frac{r}{q_1})^{q_0}\d s\\
	&\lesssim\int_0^{t^{q_1}}s^{\frac{1}{p_1}-1}\left[\left(y^{-\frac{rq_0}{p_0q_1}}\int_0^{y^\frac{rq_0}{p_0q_1}} f^*(\tau^\frac{p_0}{q_0})^{q_0}\d \tau\right)^\frac{q_1-q_0}{q_0}y^{\frac{rq_0}{p_0q_1}\cdot\frac{q_1-q_0}{q_0}}\cdot y^{\frac{rq_0}{q_1p_0}-\frac{1}{p_1}}f^*(y^\frac{r}{q_1})^{q_0}\right]^*(s)\d s\\
	&=\int_0^{t^{q_1}}s^{\frac{1}{p_1}-1}\left[\left(y^{-\frac{rq_0}{p_0}}\int_0^{y^\frac{rq_0}{p_0}} f^*(\tau^\frac{p_0}{q_0})^{q_0}\d \tau\right)^\frac{q_1-q_0}{q_0}y^{\frac{rq_0}{p_0}\cdot\frac{q_1-q_0}{q_0}}\cdot y^{\frac{rq_0}{p_0}-\frac{q_1}{p_1}}f^*(y^r)^{q_0}\right]^*(s^\frac{1}{q_1})\d s\\
	&\approx \int_0^{t^{q_1}}s^{\frac{1}{p_1}-1}\left(Yf\right)^*(s^\frac{1}{q_1})^{q_1}\d s\approx \int_0^t s^{\frac{q_1}{p_1}-1}(Yf)^*(s)^{q_1}\d s.
\end{align*} 
Here the “$\lesssim$” is \cite[Lemma 3.1. (ii)]{GoPi2009} with 
\begin{equation*}
    \beta =\frac{1}{p_1}-1,\quad \alpha=\frac{rp_1-p_0}{p_0p_1},\quad\text{and}\quad h(y)=\left(y^{-\frac{rq_0}{p_0q_1}}\int_0^{y^\frac{rq_0}{p_0q_1}} f^*(\tau^\frac{p_0}{q_0})^{q_0}\d \tau\right)^\frac{q_1-q_0}{q_0}f^*(y^\frac{r}{q_1})^{q_0}
\end{equation*} 
in their notation. The last equality takes advantage of the fact that $f_\gamma(t)^*=f^*(t^\gamma)$, where $f_\gamma(t)=f(t^\gamma)$ with $\gamma=q_1$. Hence the claim \eqref{eq:claim1} is proved. Plugging this information back into \eqref{eq:4.2.} we have
\begin{equation*}
    \begin{split}
        \|R_{q_0}f\|_{\bar Y(0, \nu(\S))}&\leq\left\|t^{-\frac{1}{p_1}}\left(\int_0^t s^{\frac{q_1}{p_1}-1}\sup_{y\geq s}Yf(y)^{q_1}\d s\right)^\frac{1}{q_1}\right\|_{\bar Y(0, \nu(\S))}\\
    	&\lesssim\left\|t^{-\frac{1}{p_1}}\left( \int_0^t s^{\frac{q_1}{p_1}-1}(Yf)^*(s)^{q_1}\d s\right)^\frac{1}{q_1}\right\|_{\bar Y(0, \nu(\S))}=\|Yf\|_{\bar Y^{\langle p_1, q_1\rangle}(0, \nu(\S))}\lesssim \|f\|_{\bar X(0, \mu(\R))}.
    \end{split}
\end{equation*}
\end{proof}

\section*{Acknowledgments.}
The author would like to express his gratitude to Zdeněk Mihula and Luboš Pick for their many insightful comments.

\bibliographystyle{dabbrv}
\bibliography{bibliography}

\end{document}